\IfFileExists{proc-l.cls}{\documentclass{proc-l}}{\documentclass{amsproc}}

\usepackage[T1]{fontenc}
\usepackage[utf8]{inputenc}
\usepackage{amsmath,amssymb,amsthm,mathtools}
\usepackage[colorlinks=true,linkcolor=blue,citecolor=blue,urlcolor=blue]{hyperref}
\allowdisplaybreaks
\newtheorem{theorem}{Theorem}[section]
\newtheorem{proposition}[theorem]{Proposition}
\newtheorem{lemma}[theorem]{Lemma}

\theoremstyle{definition}

\theoremstyle{remark}

\newtheorem{example}[theorem]{Example}

\newcommand{\F}{\mathbb F}

\newcommand{\Mcal}{\mathcal M}
\newcommand{\Pcal}{\mathcal P}
\newcommand{\Rcal}{\mathcal R}

\newcommand{\CT}{\operatorname{CT}}
\newcommand{\rank}{\operatorname{rank}}
\newcommand{\ord}{\operatorname{ord}}

\newcommand{\rad}{\operatorname{rad}}

\title[Banded quadratic digit functions along irreducibles]{Banded quadratic digit functions along irreducible polynomials over finite fields}

\author{Kaimin Cheng}
\address{School of Mathematics and Information, China West Normal University, Nanchong, 637002, P. R. China}
\email{ckm20@126.com}

\date{ }

\begin{document}

\begin{abstract}
Let $q$ be an odd prime power and let $\F_q$ be the finite field with $q$
elements.  Let $\Pcal(n)$ be the set of monic irreducible polynomials of degree
$n$ over $\F_q$.  For
$f=t^n+f_{n-1}t^{n-1}+\cdots+f_0\in\Pcal(n)$, fix coefficients
$c_0,\ldots,c_m\in\F_q$ with $c_m\ne0$ and put
\[
        Q_A(f)=\sum_{j=0}^m c_j\sum_{i=j}^n f_i f_{i-j}+\ell_n(f),
\]
where $\ell_n$ is an arbitrary linear form in the coefficients of $f$ and
$f_n=1$.  We prove that $Q_A$ is equidistributed on $\Pcal(n)$: for every
$\gamma\in\F_q$,
\[
        \#\{f\in\Pcal(n):Q_A(f)=\gamma\}
        =\frac{\#\Pcal(n)}{q}+O_A(q^{19n/20+o(n)}),
\]
as \(n\to\infty\), with \(q\) and the quadratic band fixed. This extends the finite-field Rudin--Shapiro result from nearest-neighbour correlations to arbitrary fixed symmetric Laurent symbols. The proof combines Vaughan's identity with rank estimates for Toeplitz forms; the main new
ingredient is an averaged rank-defect estimate for reciprocal symbols in the
central Type I range.
\end{abstract}

\maketitle

\section{Introduction}

Digital questions in arithmetic ask how restrictions on the digits of an
integer, or on the coefficients of a polynomial, interact with multiplicative
sets such as primes.  In the integer setting this theme goes back to problems
of Gelfond on the sum of digits \cite{Gelfond,RivatGelfond}.  A landmark result
of Mauduit and Rivat is the prime number theorem for the sum-of-digits function
\cite{MauduitRivat}; they also treated the classical Rudin--Shapiro sequence
along the primes \cite{MauduitRivatRS}.  Over a finite field, the polynomial
ring $\F_q[t]$ gives a parallel setting: the coefficients of a polynomial are
its digits, and monic irreducible polynomials play the role of primes.  Related
coefficient-prescription and digital-distribution problems have been studied in
\cite{Pollack,Ha,Porritt,Swaenepoel,DartygeSarkozy,DartygeMeraiWinterhof,MakhulWinterhof}.

The closest predecessor of this paper is M\'erai's recent work on the
finite-field Rudin--Shapiro function along irreducible polynomials
\cite{Merai}.  For
\[
        R(f)=\sum_{i=1}^{n-1}f_i f_{i-1},
\]
M\'erai proved asymptotic equidistribution on monic irreducibles, with a
Vaughan-type exponent $27/28$.  The present paper considers the corresponding
question for an arbitrary fixed finite band of quadratic coefficient
correlations.  The novelty is concentrated in three points.  First, the
quadratic symbol is an arbitrary fixed Laurent polynomial of finite bandwidth,
rather than only the nearest-neighbour symbol.  Secondly, the Type II estimate
is obtained by a Laurent-symbol rank argument.  Thirdly, the Type I estimate is
split according to the degree of the fixed multiplier.  The two outer ranges are
handled by direct rank bounds.  In the remaining central range we prove an
averaged rank-defect estimate by enlarging the symbols $Pgg^*$ to the full
linear space of reciprocal polynomials and counting a bilinear incidence.  This
linearization is the device that replaces any pointwise rank estimate in the
central range.

Let $\F_q$ be the finite field of $q$ elements, with $q$ odd.  Let $\Mcal(n)$
denote the set of monic polynomials of degree $n$, and let $\Pcal(n)$ denote
the set of monic irreducible polynomials of degree $n$.  For a polynomial
$f=f_nt^n+\cdots+f_0$ of degree at most $n$, and for $j\ge0$, define
\[
        S^{(j)}(f)=\sum_{i=j}^{n}f_if_{i-j},
\]
where missing coefficients are interpreted as zero.  When
$f\in\Mcal(n)$ or $f\in\Pcal(n)$, one has $f_n=1$.
Fix an integer $m\ge0$ and coefficients $c_0,\ldots,c_m\in\F_q$ with
$c_m\ne0$.  For each degree $n$ we allow an arbitrary linear digit form
$\ell_n$ in the coefficients $f_0,\ldots,f_n$; the estimates below are uniform
in this choice.  This uniformity is useful in applications and follows from the
fact that the linear part only changes the linear phase of the exponential sums,
never their polar ranks.  To lighten notation we write $\ell$ for $\ell_n$ and set
\begin{equation}\label{eq:Qdef}
        Q_A(f)=\sum_{j=0}^{m}c_j S^{(j)}(f)+\ell(f).
\end{equation}
The quadratic part is encoded by the symmetric Laurent polynomial
\begin{equation}\label{eq:Adef}
        A(z)=c_0+\frac12\sum_{\ell=1}^{m}c_\ell(z^\ell+z^{-\ell})\in\F_q[z,z^{-1}],
\end{equation}
and we put
\begin{equation}\label{eq:Pdef}
        P(z)=z^mA(z)\in\F_q[z].
\end{equation}
Throughout the paper $A$, equivalently $P$, is fixed and satisfies
\[
        P(0)\ne0,\qquad \deg P=2m.
\]

\begin{theorem}\label{thm:main}
For every $\gamma\in\F_q$, and uniformly in the linear digit form $\ell_n$,
\[
        \#\{f\in\Pcal(n):Q_A(f)=\gamma\}
        =\frac{\#\Pcal(n)}q+O_A\bigl(q^{19n/20+o(n)}\bigr),
\]
as $n\to\infty$ with $q$ and $A$ fixed.
\end{theorem}

We now describe the proof in a little more detail.  Vaughan's identity over
$\F_q[t]$ reduces the von Mangoldt-weighted exponential sum
\[
        \sum_{f\in\Mcal(n)}\Lambda_{\rm vM}(f)\psi(Q_A(f))
\]
to the two sums displayed in Lemma \ref{lem:vaughan} below.  Following the
standard terminology, we call them the Type I and Type II sums.  The Type II
sum is bilinear: after Cauchy's inequality it involves expressions of the form
\[
        \sum_h \psi(Q_A(hg_1)-Q_A(hg_2)).
\]
Its rank is controlled by the Laurent symbol
\[
        A(z)(g_1(z)g_1(z^{-1})-g_2(z)g_2(z^{-1}));
\]
a non-zero extreme coefficient gives a triangular minor.  The only exceptional
case is the reciprocal equation $g_1^*g_1=g_2^*g_2$, which is handled by the
usual divisor bound in $\F_q[t]$.

The Type I contribution is the one-factor part of the
Vaughan decomposition. For small multiplier
degree it is handled by a pointwise rank bound for the quadratic form
$h\mapsto Q_A(gh)$, and for large multiplier degree it is handled by a Cauchy
argument in which the dual variable is $g$.  These two arguments cover the
ranges
\[
        k\le 9n/20+O_A(1)\qquad\hbox{and}\qquad
        k\ge 3n/5-O_A(1).
\]
The only remaining range is the central band
$9n/20<k<3n/5$.  In this range a pointwise rank estimate is not strong enough.
Instead we average the radical dimensions directly.  The key observation is
that, after the zero factor of $g$ is removed, the radical condition depends on
$g$ only through the reciprocal polynomial $Pgg^*$.  We enlarge these symbols to
all reciprocal polynomials of the same degree.  For a fixed auxiliary polynomial
$H$, the condition that the middle coefficients of $TH$ vanish is then linear in
the reciprocal symbol $T$.  A simple endpoint-rank lemma for this linear map
shows that the total incidence is small enough to control
$\sum_g q^{\Delta_A(g;N)/2}$ throughout the central band.

\section{Preliminaries and quadratic sums}\label{sec:prelim}

Throughout the paper $q$ and $A$ are fixed.  All implied constants may depend on this fixed field and on $A$; we write $O_A$ for simplicity.  The notation $q^{o(n)}$ is used uniformly for all auxiliary integers in the ranges occurring below, in particular $0\le k,N\le n+O_A(1)$ and for the Vaughan cutoffs $u,v$.  Polynomial factors in $n$ and constants depending only on $A$ are absorbed into $q^{o(n)}$.

Let
\[
        V_N=\{h\in\F_q[t]:\deg h\le N\}.
\]
For a polynomial $a=a_0+\cdots+a_kt^k$ with $a_k\ne0$, write
\[
        a^*(t)=t^ka(t^{-1}).
\]
We write \([z^j]F(z)\) for the coefficient of \(z^j\) in \(F(z)\), and
\(\CT F\) for its constant term. The quadratic part of $Q_A$ can be written as
\begin{align}\label{quadratic_part} Q_{A,2}(f)=\CT\,A(z)f(z)f(z^{-1}).
\end{align}
We also use the half-polar form
\begin{equation}\label{eq:polar}
        B_A(f,h)=\CT\,A(z)f(z)h(z^{-1}).
\end{equation}
Since $A(z)=A(z^{-1})$ and $q$ is odd, the usual polar form is $2B_A$; the
factor $2$ is non-zero and has no effect on ranks.  In the sequel, all ranks are
ranks of these polar bilinear forms, and $\psi$ denotes a non-trivial additive
character of $\F_q$.  All estimates are uniform in the choice of non-trivial
$\psi$: replacing $\psi$ by $\psi(a\,\cdot)$ with $a\in\F_q^\times$ only
multiplies the quadratic and linear phases by $a$ and does not change any of
the relevant ranks.

\begin{lemma}\label{lem:quad}
Let $Q$ be a quadratic form on $\F_q^N$ with polar rank at least $r$, and let
$M$ be linear. Then
\[
\left|\sum_{x\in\F_q^N}\psi(Q(x)+M(x))\right|\le q^{N-r/2}.    
\]
\end{lemma}

\begin{proof}
Put
\[        S=\sum_{x\in\F_q^N}\psi(Q(x)+M(x))
\]
and let $B$ be the half-polar form, so that
$Q(x+h)-Q(x)=Q(h)+2B(x,h)$.  Let
\[
        \rad(B)=\{x\in\F_q^N: B(x,y)=0\ \text{for all }y\in\F_q^N\}
\]
be the radical of $B$.  Since $q$ is odd, multiplication by $2$ does not
change $\rad(B)$. Thus, we have
\begin{align*}
 |S|^2
&=\sum_{h\in\F_q^N}\sum_{x\in\F_q^N}      \psi(Q(x+h)+M(x+h)-Q(x)-M(x))  \\ &=\sum_{h\in\F_q^N}\psi(Q(h)+M(h))      \sum_{x\in\F_q^N}\psi(2B(x,h)).
\end{align*}
For fixed $h$, the inner sum is the complete sum of the linear form
$x\mapsto 2B(x,h)$.  It is equal to $q^N$ if $h\in\rad(B)$ and to $0$ otherwise.
The rank hypothesis gives $\dim\rad(B)\le N-r$.  Hence
\[
        |S|^2\le q^N\#\rad(B)\le q^{2N-r},
\]
and the desired estimate follows after taking square roots.
\end{proof}

\begin{lemma}\label{lem:monic-slice}
Let $Q$ be a quadratic form on $V_N$ of polar rank $r$.  On any affine hyperplane of $V_N$, and in particular on $\Mcal(N)$, the restricted polar rank is at least $r-2$.  Consequently
\[
        \left|\sum_{h\in\Mcal(N)}\psi(Q(h)+M(h))\right|
        \ll q^{N-r/2+O(1)}.
\]
\end{lemma}

\begin{proof}
Let \(B\) be the polar form of \(Q\) on the ambient vector space \(V_N\), and let
\(W\) be the direction space of the affine hyperplane. Thus \(W\) has
codimension one in \(V_N\). The radical of the restricted form \(B|_W\) is
\[
        \rad(B|_W)=\{w\in W:B(w,x)=0\ \text{for all }x\in W\}.
\]
If \(w\in\rad(B|_W)\), then the linear functional $x\mapsto B(w,x)$
on \(V_N\) vanishes on \(W\). Hence it belongs to the annihilator
\[
        W^\perp=\{\lambda\in V_N^\ast:\lambda|_W=0\}.
\]
Since \(W\) has codimension one in \(V_N\), the annihilator \(W^\perp\) is
one-dimensional. Therefore the image of the linear map
\[
        \rad(B|_W)\longrightarrow W^\perp,\qquad
        w\longmapsto B(w,\cdot),
\]
has dimension at most one. Its kernel is precisely \(\rad(B)\cap W\).
Consequently
\[
        \dim\rad(B|_W)
        \le \dim(\rad(B)\cap W)+1
        \le \dim\rad(B)+1.
\]
Since \(\dim W=\dim V_N-1\), it follows that the rank of the polar form
drops by at most \(2\) after restriction to \(W\). Hence the restricted
polar rank is at least \(r-2\).

Now write the affine hyperplane as \(h_0+W\).  For \(h=h_0+w\), with
\(w\in W\), we have
\[        Q(h_0+w)+M(h_0+w)
=Q_W(w)+\widetilde M(w)+c,
\]
where \(Q_W(w)\) is the restriction of $Q$ to the direction space \(W\), \(\widetilde M\) is a linear form on \(W\), and
\(c\in\mathbb F_q\) is constant. The polar form
of \(Q_W\) is \(B|_W\), and hence has rank at least \(r-2\). Applying
Lemma~\ref{lem:quad} on the vector space \(W\), which has dimension \(N\),
gives
\begin{align*}\left|\sum_{h\in\Mcal(N)}\psi(Q(h)+M(h))\right|
        &=
        \left|\sum_{w\in W}\psi(Q_W(w)+M_W(w)+c)\right|
        \\
        &\le q^{N-(r-2)/2}
        =
        q^{N-r/2+1}.
\end{align*}
This is the asserted bound.
\end{proof}

We use the standard divisor estimate
\begin{equation}\label{eq:divisor}
        \tau(F)\le q^{o(\deg F)}
\end{equation}
for the number of monic divisors of $F\in\F_q[t]$, uniformly as $\deg F\to\infty$ with $q$ fixed.  This follows from the Euler product for polynomial divisors, or from the usual maximal-order estimate in $\F_q[t]$; see, for instance, the standard irreducible factor estimates in \cite[Ch.~3]{LidlNiederreiter}.

\section{Vaughan decomposition}\label{sec:vaughan-decomp}

Let $\Lambda_{\rm vM}$ denote the von Mangoldt function on $\F_q[t]$:
$\Lambda_{\rm vM}(f)=\deg \pi$ if $f$ is a power of a monic irreducible
polynomial $\pi$, and $\Lambda_{\rm vM}(f)=0$ otherwise. We shall first estimate the weighted sum
\[        \sum_{f\in\Mcal(n)}\Lambda_{\rm vM}(f)\Psi(f),
        \qquad |\Psi(f)|\le1,
\]
where \(\Psi:\Mcal(n)\to\mathbb C\) is an arbitrary complex-valued function
bounded by \(1\).  In the application below we take
\[        \Psi(f)=\psi(Q_A(f)).
\]
At the end we pass from the von Mangoldt weighted sum to irreducible
polynomials. The following form of Vaughan's identity separates the problem into two
quantities.  The first, $\Sigma_1$, is a linear sum with one fixed factor; the
second, $\Sigma_2$, is a bilinear correlation sum with two factors.  We use the
traditional names Type I and Type II for these two estimates.

\begin{lemma}\label{lem:vaughan}
Let $|\Psi(f)|\le1$.  If $u+v<n$, then
\[
\sum_{f\in\Mcal(n)}\Lambda_{\rm vM}(f)\Psi(f)
\ll q^{o(n)}\Sigma_1+q^{n-(u+v)/2+o(n)}\Sigma_2^{1/2},
\]
where
\[
        \Sigma_1=
        \sum_{0\le k\le u+v}\sum_{g\in\Mcal(k)}
        \left|\sum_{h\in\Mcal(n-k)}\Psi(gh)\right|
\]
and
\[
        \Sigma_2=
        \max_{v\le i\le n-u}\max_{g_1\in\Mcal(n-i)}
        \sum_{g_2\in\Mcal(n-i)}
        \left|\sum_{h\in\Mcal(i)}\Psi(hg_1)\overline{\Psi(hg_2)}\right|.
\]
\end{lemma}

\begin{proof}
This is the standard function-field Vaughan identity in the form of
\cite[Lemma~9]{Merai}, applied to the bounded function \(\Psi\).  No special
property of \(\Psi\) is used beyond the bound \(|\Psi|\le1\).  In that identity
the Type I contribution consists of convolutions with a distinguished factor of
degree at most \(u+v\), giving the sum \(\Sigma_1\) above.  All multiplicities
and degree weights are bounded by a fixed power of \(n\), hence by \(q^{o(n)}\).

The Type II contribution is supported on factorizations \(F=GH\) with
\(v\le\deg H\le n-u\).  Applying Cauchy's inequality to the outer factor and
expanding the square gives the correlation sum \(\Sigma_2\), with the usual
factor \(q^{n-(u+v)/2}\).  The remaining polynomial factors in \(n\) are again
absorbed into \(q^{o(n)}\).  This gives the stated bound.
\end{proof}

\section{The bilinear sum \texorpdfstring{$\Sigma_2$}{Sigma2}}\label{sec:typeII}
Recall that, for a polynomial \(g\in\F_q[t]\) of degree \(k\), we write
\[
g^*(t)=t^kg(t^{-1})
\]
for its reciprocal polynomial.  We shall also use the quadratic part
\(Q_{A,2}\) of \(Q_A\), written as in \eqref{quadratic_part}.
\begin{lemma}\label{lem:typeII-rank}
Let $g_1,g_2\in\Mcal(k)$ and suppose $g_1^*g_1\ne g_2^*g_2$.
For $h\in V_i$, the quadratic form
$$h\mapsto Q_{A,2}(hg_1)-Q_{A,2}(hg_2)$$
has rank at least $\max\{0,i-k-m+1\}$.
\end{lemma}

\begin{proof}
The quadratic part \(Q_{A,2}\) is represented by the Laurent symbol \(A(z)\).
After the change of variables \(f=hg\), the corresponding symbol in the
\(h\)-variables is multiplied by \(g(z)g(z^{-1})\).  Hence the polar matrix
of
\[
        h\longmapsto Q_{A,2}(hg_1)-Q_{A,2}(hg_2)
\]
is the Toeplitz matrix with Laurent symbol
\[
        H(z)=A(z)\bigl(g_1(z)g_1(z^{-1})-g_2(z)g_2(z^{-1})\bigr).
\]
Since
\[
        g(z)g(z^{-1})=z^{-k}g^*(z)g(z),
\]
the assumption \(g_1^*g_1\ne g_2^*g_2\) implies
\[
        g_1(z)g_1(z^{-1})\ne g_2(z)g_2(z^{-1}).
\]
As \(\F_q[z,z^{-1}]\) is an integral domain and \(A\ne0\), it follows that
\(H(z)\ne0\).

Moreover, \(A\) has support contained in \([-m,m]\), while
\(g_j(z)g_j(z^{-1})\) has support contained in \([-k,k]\). Thus
\[
\text{supp}~H\subseteq [-k-m,k+m].
\]
Choose an exponent \(e\in\text{supp}~H\) which is maximal; the argument with a
minimal exponent is identical.  Write
\[
        H(z)=\sum_\nu H_\nu z^\nu.
\]
Then \(H_e\ne0\), and \(H_\nu=0\) for every \(\nu>e\).

With respect to the basis \(1,t,\ldots,t^i\) of \(V_i\), the polar matrix
\(M_H\) has entries
\[
        (M_H)_{a,b}=H_{b-a},\qquad 0\le a,b\le i.
\]
Consequently the diagonal \(b-a=e\) consists of entries equal to \(H_e\),
whereas all diagonals with \(b-a>e\) are zero.

The diagonal \(b-a=e\) contains
\[
        \#\{a:0\le a\le i,\ 0\le a+e\le i\}
\]
entries.  Since \(|e|\le k+m\), this number is at least
$i-k-m+1$ whenever the latter is positive.  Put
\[
        R=\max\{0,i-k-m+1\}.
\]
If \(R=0\), there is nothing to prove.  Otherwise choose \(R\) pairs
\((a_j,b_j)\), \(1\le j\le R\), on the diagonal \(b-a=e\), ordered so that
\(a_1<\cdots<a_R\) and \(b_j=a_j+e\).  Consider the submatrix of \(M_H\)
formed by the rows \(a_1,\ldots,a_R\) and the columns \(b_1,\ldots,b_R\).
Since entries with \(b-a>e\) vanish, this submatrix is triangular after the
chosen ordering, and its diagonal entries are all \(H_e\).  Its determinant is
therefore \(H_e^R\ne0\).

Thus \(M_H\) has a nonzero \(R\times R\) minor, and hence
\[
        \operatorname{rank} M_H\ge R
        =
        \max\{0,i-k-m+1\}.
\]
This proves Lemma \ref{lem:typeII-rank}.
\end{proof}

\begin{lemma}\label{lem:reciprocal}
For fixed $b\in\Mcal(k)$, the number of $a\in\Mcal(k)$ satisfying
\[
        a^*a=b^*b
\]
is at most $\tau(b)$, where \(\tau(b)\) denotes the number of monic divisors of \(b\).
\end{lemma}

\begin{proof}
We prove that the map
\[
        a\longmapsto d=\gcd(a,b)
\]
from the set of solutions \(a\in\Mcal(k)\) of \(a^*a=b^*b\) to the set of
monic divisors of \(b\) is injective.  Since the image lies among the monic
divisors of \(b\), this will give the desired bound \(\tau(b)\).

Suppose that \(a\) is a solution, and put $d=\gcd(a,b)$. Write
$a=da_1,\ b=db_1$ with \(a_1,b_1\) monic and \(\gcd(a_1,b_1)=1\). Since reversal is
multiplicative, the identity \(a^*a=b^*b\) gives $d^*d\,a_1^*a_1=d^*d\,b_1^*b_1$, and hence
$a_1^*a_1=b_1^*b_1$.
Thus \(a_1\mid b_1^*b_1\). Since \((a_1,b_1)=1\), we have
$ a_1\mid b_1^*$.

Since
\(\deg a=\deg b=k\) and \(\deg d\) is fixed, we have
$\deg a_1=\deg b_1$.
If \(b_1(0)=0\), then \(\deg b_1^*<\deg b_1\), so no monic polynomial
\(a_1\) of degree \(\deg b_1\) can divide \(b_1^*\). Hence \(b_1(0)\ne0\), and then \(\deg b_1^*=\deg b_1\). It follows that \(a_1\) is a
monic divisor of \(b_1^*\) of the same degree as \(b_1^*\). Therefore
\(a_1\) must be the monic normalization of \(b_1^*\). Since the leading
coefficient of \(b_1^*\) is \(b_1(0)\), this gives
$a_1=b_1(0)^{-1}b_1^*$.
Thus, 
$a=\gcd(a,b)b_1(0)^{-1}b_1^*$. Consequently two solutions \(a,a'\) with \(\gcd(a,b)=\gcd(a',b)\)
must be equal. Hence the map \(a\mapsto \gcd(a,b)\) is injective, as claimed.
\end{proof}

\begin{proposition}\label{prop:typeII}
With $\Sigma_2$ as in Lemma \ref{lem:vaughan} and $\Psi(f)=\psi(Q_A(f))$, assume $v>n/2+O_A(1)$.  Then
\[
        \Sigma_2\ll_A q^{n-u+o(n)}+q^{3n/2-v+o(n)}.
\]
\end{proposition}

\begin{proof}
Fix $g_1\in\Mcal(n-i)$ and write $k=n-i$.  If $g_1^*g_1\ne g_2^*g_2$, Lemmas \ref{lem:typeII-rank} and \ref{lem:monic-slice} give
\[
        \left|\sum_{h\in\Mcal(i)}\psi(Q_A(hg_1)-Q_A(hg_2))\right|
        \ll_A q^{i-(i-k-m)/2+O(1)}.
\]
Since $i\ge v>n/2+O_A(1)$, the last bound is $\ll_A q^{n/2+o(n)}$.  Summing over the non-exceptional $g_2$ gives
\[
        q^{n-i}\,q^{n/2+o(n)}.
\]
For the exceptional $g_2$, Lemma \ref{lem:reciprocal} and \eqref{eq:divisor} give only $q^{o(n)}$ choices, and the inner sum is trivially $O(q^i)$.  Hence the exceptional contribution is $q^{i+o(n)}\le q^{n-u+o(n)}$.  Taking the maximum over $v\le i\le n-u$ gives
\[
        \Sigma_2\ll_A q^{n-u+o(n)}+q^{3n/2-v+o(n)}.
\]
\end{proof}

\section{Rank defects for the linear sums}\label{sec:typeI}

For $0\ne g\in\F_q[z]$ and $N\ge0$ let
\[
        B_{g,N}(h,w)=B_A(gh,gw)\qquad (h,w\in V_N)
\]
be the polar form of the quadratic part of $h\mapsto Q_A(gh)$ on $V_N$.  We
write
\begin{equation}\label{eq:Delta-def}
        \Delta_A(g;N)=\dim\rad(B_{g,N})=N+1-\rank B_{g,N}.
\end{equation}
The linear part of $Q_A$ does not affect this definition.

The only technical nuisance in the linear sums is a possible factor $z^r$ in
$g$.  This factor disappears from the Laurent symbol.  Write
\[
        r(g)=\ord_z g,\qquad g_0(z)=z^{-r(g)}g(z),\qquad k_0=\deg g_0 .
\]
Thus $g_0(0)\ne0$.  Put
\[
        d_0=k_0+m,\qquad F_{g_0}(z)=P(z)g_0(z)g_0^*(z).
\]
Then $F_{g_0}(0)\ne0$ and $\deg F_{g_0}=2d_0$.  Moreover
\[
        g(z)g(z^{-1})=g_0(z)g_0(z^{-1}),
\]
and hence
\begin{equation}\label{eq:Delta-zero-removal}
        \Delta_A(g;N)=\Delta_A(g_0;N).
\end{equation}

\begin{lemma}\label{lem:gap}
Let $g=z^r g_0$ with $g_0(0)\ne0$ and let $d_0=\deg g_0+m$.  For
$h\in V_N$, $h$ lies in the radical of $B_{g,N}$ if and only if
\[
        [z^j]F_{g_0}(z)h(z)=0\qquad(d_0\le j\le d_0+N).
\]
\end{lemma}

\begin{proof}
For $w\in V_N$ we have
\begin{align*}
        B_A(gh,gw)
        &=\CT\,A(z)g(z)h(z)g(z^{-1})w(z^{-1})\\
        &=\CT\,A(z)g_0(z)h(z)g_0(z^{-1})w(z^{-1}),
\end{align*}
where the power $z^r$ has cancelled against $z^{-r}$.  Since
$A(z)=z^{-m}P(z)$ and $g_0(z^{-1})=z^{-k_0}g_0^*(z)$, this becomes
\[
        \CT\,z^{-d_0}F_{g_0}(z)h(z)w(z^{-1}).
\]
Write $w(z)=\sum_{a=0}^N w_a z^a$.  Then
$w(z^{-1})=\sum_{a=0}^N w_a z^{-a}$, and the last constant term is
\[
        \sum_{a=0}^N w_a\,[z^{d_0+a}]F_{g_0}(z)h(z).
\]
This expression is zero for every choice of $w_0,\ldots,w_N$ if and only if
all the displayed coefficients vanish, namely
\[
        [z^j]F_{g_0}(z)h(z)=0
        \qquad(d_0\le j\le d_0+N).
\]
This is precisely the stated radical condition.
\end{proof}

\section{A central averaged rank-defect estimate}\label{sec:global}

In the central range it is useful to forget the special factorization
$Pgg^*$ and to average over the whole linear space of reciprocal symbols.  This
gives a weaker symbol estimate than a full layer estimate would give, but it is
sufficient for the Vaughan decomposition after the Type I range has been split.

Let
\[
        \Mcal^\times(k)=\{g\in\Mcal(k):g(0)\ne0\}.
\]
For $d\ge0$ let $\Rcal_d$ be the $\F_q$-vector space of reciprocal polynomials
with centre $d$:
\[
        \Rcal_d=\{T(z)=\sum_{s=0}^{2d}T_s z^s:T_s=T_{2d-s}\}.
\]
Thus $\dim \Rcal_d=d+1$.  For $T\in\Rcal_d$ and $N\ge0$ define
\[
        \Delta(T;N)=\dim\{H\in V_N:[z^e]T(z)H(z)=0\quad(d\le e\le d+N)\}.
\]
Because $P^*=P$, if $g\in\Mcal^\times(\kappa)$, $d=\kappa+m$, and
$T=Pgg^*$, then $T\in\Rcal_d$.  Lemma~\ref{lem:gap} gives
\begin{equation}\label{eq:Delta-symbol-equality}
        \Delta_A(g;N)=\Delta(T;N).
\end{equation}
Thus a half-moment bound for \(q^{\Delta(T;N)}\) over \(\Rcal_d\) will imply
the estimate
\[
        \sum_{g\in\Mcal^\times(\kappa)}q^{\Delta_A(g;N)/2}
        \ll_A q^{\kappa+3N/8+o(n)}.
\]
Combined with the quadratic-sum bound on the monic slice, this gives the
central Type I exponent \(q^{\kappa+7N/8+o(n)}\).

We need only a one-sided endpoint-rank estimate for multiplication by a fixed
polynomial on the reciprocal-symbol space.

\begin{lemma}\label{lem:reciprocal-symbol-rank}
Let $C\ge0$ be fixed.  Let $H\in V_N$ be non-zero and write
\[
        H=z^\nu(h_0+h_1z+\cdots+h_Rz^R),\qquad h_0h_R\ne0.
\]
Assume $R\le 4d/3+C$.  Consider the linear map
\[
        L_H:\Rcal_d\longrightarrow \F_q^{N+1},\qquad
        T\longmapsto ([z^{d+j}]T(z)H(z))_{0\le j\le N}.
\]
Then
\[
        \rank L_H\ge \frac R4-O_C(1).
\]
\end{lemma}

\begin{proof}
First remove the initial power of $z$.  If
$H=z^\nu H_0$ with $\deg H_0=R$, then
\[
        [z^{d+\nu+j}]T(z)H(z)=[z^{d+j}]T(z)H_0(z).
\]
Because $\nu+R=\deg H\le N$, all rows with shifted indices
$\nu,\nu+1,\ldots,\nu+R$ are among the rows of $L_H$.  It is therefore enough to
prove the claimed lower bound for $H_0$.  Thus we may assume from now on that
\[
        H=h_0+h_1z+\cdots+h_Rz^R,
        \qquad h_0h_R\ne0 .
\]

Use the coordinates of $\Rcal_d$ given by
\[
        T(z)=x_0z^d+\sum_{a=1}^{d}x_a(z^{d+a}+z^{d-a}).
\]
For $0\le j\le R$, define the row functional
\[
        C_j=[z^{d+j}]T(z)H(z).
\]
A direct coefficient calculation gives
\[
        C_j=\sum_{s=0}^{R}h_sx_{|j-s|},
        \qquad x_a:=0\quad(a>d).
\]
We now choose rows for which the endpoint coefficient $h_R$ produces a pivot.
Let
\[
        J=\{j\in\mathbb Z:\max(0,R-d)\le j\le \lfloor(R-1)/2\rfloor\}.
\]
For $j\in J$ we have $0\le R-j\le d$, so the variable $x_{R-j}$ is one of the
coordinates of $\Rcal_d$.  In the row $C_j$ it appears from the term $s=R$ with
coefficient $h_R\ne0$.  For every other $s<R$, the index of the corresponding
variable is strictly smaller:
\[
        |j-s|\le \max(j,R-1-j)<R-j,
\]
because $j<R/2$.  Thus, if the columns are ordered by decreasing index, the
submatrix with rows $C_j$ for $j\in J$ and columns $x_{R-j}$ for $j\in J$ is
triangular with diagonal entries all equal to $h_R$.  It has full rank $|J|$.
Consequently
\[
        \rank L_H\ge |J|.
\]
It remains only to estimate the size of $J$.  If $R\le d$, then
\[
        |J|\ge \lfloor(R-1)/2\rfloor+1\ge R/2-O(1).
\]
If $d<R\le 4d/3+C$, then
\[
        |J|\ge \lfloor(R-1)/2\rfloor-(R-d)+1
             = d-R/2-O(1)
             \ge R/4-O_C(1).
\]
The case $R=0$ gives a vacuous lower bound.  This proves the lemma.
\end{proof}

\begin{proposition}\label{prop:reciprocal-incidence}
Uniformly for $d,N$ with $N\le (4/3)d+O_A(1)$,
\[
        \sum_{T\in\Rcal_d} q^{\Delta(T;N)}
        \ll_A q^{d+3N/4+o(n)}.
\]
Consequently,
\[
        \sum_{T\in\Rcal_d} q^{\Delta(T;N)/2}
        \ll_A q^{d+3N/8+o(n)}.
\]
\end{proposition}

\begin{proof}
For a fixed $T\in\Rcal_d$, the quantity $q^{\Delta(T;N)}$ is exactly the number
of polynomials $H\in V_N$ satisfying
\[
        [z^e]T(z)H(z)=0\qquad(d\le e\le d+N).
\]
Therefore
\[
        \sum_{T\in\Rcal_d}q^{\Delta(T;N)}
\]
counts the pairs $(T,H)\in\Rcal_d\times V_N$ satisfying these middle-coefficient
equations.  We estimate this incidence by summing first over $H$.

If $H=0$, every $T\in\Rcal_d$ is allowed.  This gives $\#\Rcal_d=q^{d+1}$, which is $O(q^{d+3N/4+O(1)})$ and is admissible.

Assume now that $H\ne0$.  Write $\nu=\ord_zH$, $\lambda=\deg H$, and
$R=\lambda-\nu$.  For fixed $\nu$ and $\lambda$, the polynomial $H$ has the form
\[
        H=z^\nu(h_0+h_1z+\cdots+h_Rz^R),
        \qquad h_0h_R\ne0,
\]
so the number of such $H$ is at most $q^{R+1}$.  There are only $O(N^2)$ possible
pairs $(\nu,\lambda)$.

For this fixed non-zero $H$, the equations in $T$ are precisely the linear
system $L_H(T)=0$.  Since $R\le N\le(4/3)d+O_A(1)$,
Lemma~\ref{lem:reciprocal-symbol-rank} gives
\[
        \rank L_H\ge R/4-O_A(1).
\]
As $\dim\Rcal_d=d+1$, the kernel of $L_H$ has size at most
\[
        q^{d+1-R/4+O_A(1)}.
\]
Thus the total number of incidences with this fixed pair $(\nu,\lambda)$ is
\[
        \ll_A q^{R+1}\,q^{d+1-R/4+O_A(1)}
        \ll_A q^{d+3R/4+O_A(1)}
        \le q^{d+3N/4+O_A(1)}.
\]
After summing over the $O(N^2)$ choices of $(\nu,\lambda)$, the polynomial
factor is absorbed into $q^{o(n)}$.  This proves
\[
        \sum_{T\in\Rcal_d}q^{\Delta(T;N)}
        \ll_A q^{d+3N/4+o(n)}.
\]

The half-moment follows from Cauchy's inequality and
$\#\Rcal_d=q^{d+1}$:
\[
\sum_{T\in\Rcal_d}q^{\Delta(T;N)/2}
\le \#\Rcal_d^{1/2}
     \left(\sum_{T\in\Rcal_d}q^{\Delta(T;N)}\right)^{1/2}
\ll_A q^{d+3N/8+o(n)}.
\]
\end{proof}

\begin{proposition}\label{prop:central-half}
Uniformly for all non-negative integers $\kappa,N$ in the ranges $\kappa,N\le n+O_A(1)$ and satisfying
\[
        N\le \frac43\kappa+O_A(1),
\]
one has
\[
        \sum_{g\in\Mcal^\times(\kappa)} q^{\Delta_A(g;N)/2}
        \ll_A q^{\kappa+3N/8+o(n)} .
\]
\end{proposition}

\begin{proof}
Put $d=\kappa+m$.  Then the hypothesis implies
$N\le(4/3)d+O_A(1)$, so Proposition~\ref{prop:reciprocal-incidence} is
applicable.

For $g\in\Mcal^\times(\kappa)$, the polynomial
\[
        T=Pgg^*
\]
is reciprocal of centre $d$; indeed $P^*=P$ and $(gg^*)^*=gg^*$.  Hence
$T\in\Rcal_d$, and Lemma~\ref{lem:gap} gives
$\Delta_A(g;N)=\Delta(T;N)$.  We next bound the multiplicity of the map
$g\mapsto Pgg^*$.  Fix $T\in\Rcal_d$.  If \(T=0\), then there is no preimage,
because \(Pgg^*\ne0\).  Assume henceforth that \(T\ne0\).  If $P\nmid T$, there
is no preimage.  If $P\mid T$, then any preimage $g$ satisfies
\[
        g\mid gg^*=T/P.
\]
Thus $g$ is a monic divisor of the fixed non-zero polynomial $T/P$.  The divisor
estimate \eqref{eq:divisor}, uniformly in degree $O(n)$, gives at most
$q^{o(n)}$ such preimages.  Therefore
\[
\sum_{g\in\Mcal^\times(\kappa)}q^{\Delta_A(g;N)/2}
\le q^{o(n)}\sum_{T\in\Rcal_d}q^{\Delta(T;N)/2}
\ll_A q^{d+3N/8+o(n)}.
\]
Since $d=\kappa+m=\kappa+O_A(1)$, this is the desired estimate.
\end{proof}

\section{Completion of the Vaughan estimate}\label{sec:vaughan}

For $0\le k\le n$ put $N=n-k$ and
\[
        T_k=
        \sum_{g\in\Mcal(k)}
        \left|\sum_{h\in\Mcal(N)}\psi(Q_A(gh))\right| .
\]
Then $\Sigma_1=\sum_{0\le k\le u+v}T_k$.

\begin{lemma}\label{lem:typeI-pointwise-rank}
Let $g\in\Mcal(k)$ be non-zero.  For $h\in V_N$, the quadratic form
$h\mapsto Q_{A,2}(gh)$ has polar rank at least
\[
        \max\{0,N-k-m+1\}.
\]
Consequently
\[
        \left|\sum_{h\in\Mcal(N)}\psi(Q_A(gh))\right|
        \ll_A q^{N-\max(0,N-k-m)/2+O_A(1)} .
\]
\end{lemma}

\begin{proof}
The half-polar form of $h\mapsto Q_{A,2}(gh)$ is
\[
        (h,w)\longmapsto \CT\,A(z)g(z)g(z^{-1})h(z)w(z^{-1}).
\]
Let
\[
        H(z)=A(z)g(z)g(z^{-1}).
\]
This Laurent polynomial is non-zero and is supported in $[-k-m,k+m]$.  Write
$D=k+m$.  Choose an exponent $e$ at one end of the support with coefficient
$H_e\ne0$; for definiteness take the largest such exponent.  In the coefficient
matrix with respect to the basis $1,t,\ldots,t^N$, the entry in row $a$ and
column $b$ is $H_{a-b}$ up to the harmless transpose convention.  Hence one of
the diagonals $a-b=e$ or $b-a=e$ consists of the non-zero constant $H_e$, and all
more extreme parallel diagonals are zero.

Because $|e|\le D$, this diagonal meets the $(N+1)\times(N+1)$ matrix in at
least $N-D+1$ positions when $N\ge D$.  Taking the rows and columns through
these positions, ordered along the diagonal, gives a triangular submatrix with
all diagonal entries equal to $H_e$.  Hence the polar rank on $V_N$ is at least
$N-D+1=N-k-m+1$ when this number is positive, and the asserted lower bound is
trivial otherwise.  Finally, Lemma~\ref{lem:monic-slice} transfers the estimate
to the monic affine slice; the linear part of $Q_A$ only adds a linear phase and
therefore does not change the polar rank.
\end{proof}

\begin{proposition}\label{prop:typeI-small-range}
Uniformly for $0\le k\le 9n/20+O_A(1)$,
\[
        T_k\ll_A q^{19n/20+o(n)} .
\]
\end{proposition}

\begin{proof}
In this range $N-k=n-2k\ge n/10-O_A(1)$, so the positive part in
Lemma~\ref{lem:typeI-pointwise-rank} is active for all sufficiently large $n$.
For each $g\in\Mcal(k)$,
\[
        \left|\sum_{h\in\Mcal(N)}\psi(Q_A(gh))\right|
        \ll_A q^{N-(N-k-m)/2+O_A(1)}
        \ll_A q^{N/2+k/2+O_A(1)} .
\]
There are exactly $q^k$ monic polynomials of degree $k$.  Therefore
\[
        T_k\ll_A q^k q^{N/2+k/2+O_A(1)}
        =q^{N/2+3k/2+O_A(1)}
        =q^{n/2+k+O_A(1)} .
\]
If $k\le9n/20+O_A(1)$, then $n/2+k\le19n/20+O_A(1)$, and the claimed bound
follows.
\end{proof}

\begin{proposition}\label{prop:typeI-large-range}
Uniformly for $3n/5-O_A(1)\le k\le 9n/10+O_A(1)$,
\[
        T_k\ll_A q^{19n/20+o(n)} .
\]
\end{proposition}

\begin{proof}
Put
\[
        S_g=\sum_{h\in\Mcal(N)}\psi(Q_A(gh)).
\]
Cauchy's inequality gives
\[
        T_k^2=\left(\sum_{g\in\Mcal(k)}|S_g|\right)^2
        \le q^k\sum_{g\in\Mcal(k)}|S_g|^2 .
\]
Expanding the second moment, and then interchanging the order of summation,
we obtain
\[
\sum_{g\in\Mcal(k)}|S_g|^2
=\sum_{h_1,h_2\in\Mcal(N)}
  \sum_{g\in\Mcal(k)}
  \psi(Q_A(gh_1)-Q_A(gh_2)).
\]
For a pair $(h_1,h_2)$ with $h_1^*h_1\ne h_2^*h_2$, the quadratic part in the
variable $g$ has Laurent symbol
\[
        A(z)\bigl(h_1(z)h_1(z^{-1})-h_2(z)h_2(z^{-1})\bigr).
\]
This symbol is non-zero and is supported in $[-N-m,N+m]$.  Applying the same
extreme-diagonal minor argument as in Lemma~\ref{lem:typeI-pointwise-rank}, but
now on the space of polynomials $g$ of degree at most $k$, gives rank at least
$\max\{0,k-N-m+1\}$.  Restricting to the monic slice in $g$ and using
Lemma~\ref{lem:monic-slice} yields
\[
        \left|\sum_{g\in\Mcal(k)}
        \psi(Q_A(gh_1)-Q_A(gh_2))\right|
        \ll_A q^{k-(k-N)/2+O_A(1)}
\]
for every non-exceptional pair.  Since there are $q^{2N}$ pairs
$(h_1,h_2)\in\Mcal(N)^2$, their total contribution to the second moment is
\[
        \ll_A q^{2N+k-(k-N)/2+o(n)}.
\]

For exceptional pairs satisfying $h_1^*h_1=h_2^*h_2$, Lemma~\ref{lem:reciprocal}
shows that, once $h_1$ is fixed, there are at most $\tau(h_1)=q^{o(n)}$ choices
for $h_2$.  For these pairs we use the trivial bound $q^k$ for the inner sum
over $g$.  Hence the exceptional contribution is
\[
        \ll q^N q^{o(n)}q^k=q^{N+k+o(n)}.
\]
Thus
\[
\sum_{g\in\Mcal(k)}|S_g|^2
\ll_A q^{2N+k-(k-N)/2+o(n)}+q^{N+k+o(n)} .
\]
Taking square roots after Cauchy's inequality gives
\[
        T_k
        \ll_A q^{k/2}\,q^{N+(k-(k-N)/2)/2+o(n)}
              +q^{k/2}\,q^{(N+k)/2+o(n)}.
\]
Since $N=n-k$, this is
\[
        T_k\ll_A q^{5n/4-k/2+o(n)}+q^{k+N/2+o(n)}.
\]
The first exponent is at most $19n/20+o(n)$ when
$k\ge3n/5-O_A(1)$, because $5n/4-k/2\le19n/20+O_A(1)$.  The second exponent is
$k+N/2=n/2+k/2$, which is at most $19n/20+O_A(1)$ when
$k\le9n/10+O_A(1)$.  This proves the proposition.
\end{proof}

\begin{proposition}\label{prop:typeI-central-range}
Uniformly for $9n/20-O_A(1)<k<3n/5+O_A(1)$,
\[
        T_k\ll_A q^{19n/20+o(n)} .
\]
\end{proposition}

\begin{proof}
Every $g\in\Mcal(k)$ has a unique decomposition
\[
        g=z^r g_0,
        \qquad
        g_0\in\Mcal^\times(\kappa),
        \qquad
        \kappa=k-r.
\]
The factor $z^r$ cancels in the Laurent symbol $g(z)g(z^{-1})$, and therefore
\eqref{eq:Delta-zero-removal} gives
\[
        \Delta_A(g;N)=\Delta_A(g_0;N).
\]
We estimate the contribution for a fixed value of $r$, or equivalently a fixed
$\kappa$, and then sum over the $O(n)$ possible values of $r$.

Set
\[
        \kappa_0=\frac{3n}{10}+\frac{k}{3}.
\]
We first consider the subrange $\kappa>\kappa_0-O_A(1)$.  The hypothesis
$k>9n/20-O_A(1)$ implies
\[
\kappa_0-\frac34N
 =\frac{3n}{10}+\frac{k}{3}-\frac34(n-k)
 =\frac{13k}{12}-\frac{9n}{20}
 \ge -O_A(1),
\]
because $N=n-k$.  Hence
\[
        \kappa\ge\frac34N-O_A(1),
        \qquad\hbox{so}\qquad
        N\le\frac43\kappa+O_A(1).
\]
For each $g_0$, Lemmas~\ref{lem:quad} and \ref{lem:monic-slice} give
\[
        \left|\sum_{h\in\Mcal(N)}\psi(Q_A(z^r g_0h))\right|
        \ll_A q^{N/2+\Delta_A(g_0;N)/2+O_A(1)}.
\]
Summing this estimate over $g_0\in\Mcal^\times(\kappa)$ and using
Proposition~\ref{prop:central-half}, which is applicable by the inequality just
proved, we obtain
\[
\sum_{g_0\in\Mcal^\times(\kappa)}
        \left|\sum_{h\in\Mcal(N)}\psi(Q_A(z^r g_0h))\right|
        \ll_A q^{N/2}\sum_{g_0\in\Mcal^\times(\kappa)}
              q^{\Delta_A(g_0;N)/2}
        \ll_A q^{\kappa+7N/8+o(n)}.
\]
Since $\kappa\le k$, this is at most $q^{k+7N/8+o(n)}$.  Finally
\[
        k+\frac78N=k+\frac78(n-k)=\frac78n+\frac18k
        \le \frac{19n}{20}+O_A(1)
\]
throughout $k<3n/5+O_A(1)$.  Thus this part contributes
$O_A(q^{19n/20+o(n)})$ for each fixed $r$.

It remains to consider $\kappa\le\kappa_0+O_A(1)$.  Here we use the pointwise
rank bound.  Since the zero factor $z^r$ does not change the Laurent symbol,
Lemma~\ref{lem:typeI-pointwise-rank}, with $\kappa$ in place of $k$, gives for
fixed $\kappa$
\[
\sum_{g_0\in\Mcal^\times(\kappa)}
        \left|\sum_{h\in\Mcal(N)}\psi(Q_A(z^r g_0h))\right|
\ll_A q^{\kappa+N-\max(0,N-\kappa)/2+O_A(1)} .
\]
If $\kappa\le N$, then the exponent is $N/2+3\kappa/2+O_A(1)$, and hence
\[
        N/2+3\kappa/2
        \le N/2+3\kappa_0/2+O_A(1)
        =\frac{n-k}{2}+\frac32\left(\frac{3n}{10}+\frac{k}{3}\right)+O_A(1)
        =\frac{19n}{20}+O_A(1).
\]
If $\kappa>N$, then such a $\kappa\le\kappa_0+O_A(1)$ can occur only if
\[
        \kappa_0>N-O_A(1),
\]
which is equivalent to
\[
        \frac{3n}{10}+\frac{k}{3}>n-k-O_A(1),
        \qquad\hbox{or}\qquad
        k\ge\frac{21n}{40}-O_A(1).
\]
In this case the pointwise estimate is bounded by $q^{\kappa+N+O_A(1)}$, and
\[
        \kappa+N
        \le \kappa_0+N+O_A(1)
        =\frac{3n}{10}+\frac{k}{3}+n-k+O_A(1)
        =\frac{13n}{10}-\frac{2k}{3}+O_A(1)
        \le \frac{19n}{20}+O_A(1),
\]
where the last inequality uses $k\ge21n/40-O_A(1)$.  Therefore this second part
also contributes $O_A(q^{19n/20+o(n)})$ for each fixed $r$.

Summing over the $O(n)$ possible values of $r$ only changes the estimate by a
factor absorbed into $q^{o(n)}$.
\end{proof}

\begin{proposition}\label{prop:typeI}
For the Vaughan cutoffs
\[
        u=\left\lfloor\frac n5\right\rfloor,
        \qquad
        v=\left\lfloor\frac{7n}{10}\right\rfloor,
\]
and $\Psi(f)=\psi(Q_A(f))$ with $\psi$ non-trivial, one has
\[
        \Sigma_1\ll_A q^{19n/20+o(n)} .
\]
\end{proposition}

\begin{proof}
For these cutoffs, $u+v=9n/10+O(1)$.  The three estimates
Propositions~\ref{prop:typeI-small-range}, \ref{prop:typeI-central-range}, and
\ref{prop:typeI-large-range} cover all integers $0\le k\le u+v$.  Summing the
corresponding bounds for $T_k$ over $O(n)$ values of $k$ gives the claim.
\end{proof}

\begin{proof}[Proof of Theorem \ref{thm:main}]
Let $\psi$ be a non-trivial additive character and choose
\[
        u=\left\lfloor\frac n5\right\rfloor,
        \qquad v=\left\lfloor\frac{7n}{10}\right\rfloor .
\]
Then $u+v=9n/10+O(1)<n$ for large $n$.  Set
\[
        \mathcal V_n(\psi)=
        \sum_{f\in\Mcal(n)}\Lambda_{\rm vM}(f)\psi(Q_A(f)).
\]
Lemma~\ref{lem:vaughan}, Proposition~\ref{prop:typeI}, and
Proposition~\ref{prop:typeII} give
\begin{align*}
\mathcal V_n(\psi)
&\ll_A q^{19n/20+o(n)} \\
&\quad +q^{n-(u+v)/2+o(n)}
 \left(q^{n-u}+q^{3n/2-v}\right)^{1/2}.
\end{align*}
We check the exponent of the second term.  Since
$u=n/5+O(1)$ and $v=7n/10+O(1)$,
\[
        n-\frac{u+v}{2}+\frac{n-u}{2}
        =\frac{3n}{2}-u-\frac v2+O(1)
        =\frac{19n}{20}+O(1),
\]
and
\[
        n-\frac{u+v}{2}+\frac{3n/2-v}{2}
        =\frac{7n}{4}-\frac u2-v+O(1)
        =\frac{19n}{20}+O(1).
\]
Thus
\[
        \mathcal V_n(\psi)\ll_A q^{19n/20+o(n)}.
\]
As noted above, the bound is uniform for all non-trivial additive characters
\(\psi\).

We now remove the von Mangoldt weight.  If $f\in\Pcal(n)$, then
$\Lambda_{\rm vM}(f)=n$.  A prime power of degree $n$ which is not prime has
the form $\pi^j$ with $j\ge2$ and $\deg\pi=n/j\le n/2$.  By the prime polynomial theorem over finite fields (see, e.g., \cite[Ch.~3]{LidlNiederreiter}), the number of such
prime powers is $O(q^{n/2})$, and each has von Mangoldt weight at most $n$;
hence their total contribution to the weighted sum is $O(nq^{n/2})$.  Therefore
\[
        \sum_{f\in\Mcal(n)}\Lambda_{\rm vM}(f)\psi(Q_A(f))
        =n\sum_{f\in\Pcal(n)}\psi(Q_A(f))+O(nq^{n/2}).
\]
Dividing by $n$ gives, for every non-trivial $\psi$,
\[
        \sum_{f\in\Pcal(n)}\psi(Q_A(f))
        \ll_A q^{19n/20+o(n)}.
\]
Finally use additive-character orthogonality on $\F_q$:
\[
\#\{f\in\Pcal(n):Q_A(f)=\gamma\}
=\frac1q\sum_{\psi}\psi(-\gamma)
        \sum_{f\in\Pcal(n)}\psi(Q_A(f)).
\]
The trivial character contributes $\#\Pcal(n)/q$, while the $q-1$ non-trivial
characters are bounded by the estimate just proved.  This gives the asserted
formula.
\end{proof}

\appendix
\section{A low-rank obstruction}\label{sec:lowrank}

The following example shows that endpoint-diagonal pointwise rank bounds can be essentially sharp; hence in the central Type I range such bounds alone do not deliver the required saving, and an averaged incidence estimate is needed.

\begin{example}\label{ex:lowrank}
Let $q=3$, $Q(f)=S^{(1)}(f)$, and $a(t)=t^4-1$.  For
\[
        h(t)=x_0+x_1t+\cdots+x_6t^6,
\]
the quadratic form $h\mapsto S^{(1)}(ah)$ has matrix, up to a non-zero scalar,
\[
\begin{pmatrix}
0&1&0&1&0&1&0\\
1&0&1&0&1&0&1\\
0&1&0&1&0&1&0\\
1&0&1&0&1&0&1\\
0&1&0&1&0&1&0\\
1&0&1&0&1&0&1\\
0&1&0&1&0&1&0
\end{pmatrix}.
\]
Over $\F_3$ this matrix has rank $2$.
\end{example}

\end{document}